\documentclass{amsart}

\usepackage[utf8]{inputenc}

\usepackage{amssymb}
\usepackage{amsmath}
\usepackage{verbatim}
\usepackage{url}
\usepackage{graphicx}
\usepackage{color}
\usepackage[shortlabels]{enumitem}

\usepackage[foot]{amsaddr}

\usepackage{mathrsfs}

\newcommand{\eps}{\varepsilon}
\newcommand{\jvRe}{\operatorname{Re}}
\newcommand{\jvIm}{\operatorname{Im}}
\newcommand{\Bff}{\mathbf}

\newcommand{\ii}{\textnormal{i}}
\newcommand{\dd}{\textnormal{d}}
\newcommand{\ee}{\textnormal{e}}

\newtheorem{theorem}{Theorem}
\newtheorem{proposition}[theorem]{Proposition}
\newtheorem{corollary}[theorem]{Corollary}

\newcommand\xqed[1]{%
  \leavevmode\unskip\penalty9999 \hbox{}\nobreak\hfill
  \quad\hbox{#1}}
\theoremstyle{definition}
\newtheorem{xdefinition}[theorem]{Definition}
\newenvironment{definition}{\begin{xdefinition}}{\xqed{$\triangle$}\end{xdefinition}}
\theoremstyle{remark}
\newtheorem{xremark}[theorem]{Remark}
\newenvironment{remark}{\begin{xremark}}{\xqed{$\triangle$}\end{xremark}}
\newtheorem{xexample}[theorem]{Example}
\newenvironment{example}{\begin{xexample}}{\xqed{$\triangle$}\end{xexample}}

\newtheorem*{acknowledgements*}{Acknowledgements}
\newtheorem*{notation*}{Notation}
\newtheorem*{planofpaper*}{Plan of the paper}

\newcommand{\shift}{\mathcal{S}}

\newcommand{\skewop}{\mathcal{W}}
\newcommand{\scaleop}{\mathcal{V}}
\newcommand{\barg}{\mathfrak{B}}
\newcommand{\ol}{\overline}

\title[Norms of embeddings]{Norms of embeddings between quadratically weighted spaces of holomorphic functions}

\author{Joe Viola}
\address{Nantes Université, Laboratoire de
Mathématiques Jean Leray, LMJL, F-44000 Nantes, France}
\email{joseph.viola@univ-nantes.fr}
\begin{document}

\begin{abstract}
We consider spaces of holomorphic functions which are square-integrable against a Gaussian weight, which appear in the theory of metaplectic FBI--Bargmann transforms. We identify the operator norm of embeddings between two such spaces, by relating these embeddings to Fourier integral operators with complex phase.
\end{abstract}

\maketitle

\section{Introduction and main result}\label{s:intro}

When performing phase-space analysis using FBI--Bargmann transforms, one is often led to consider the effect of changing the weight on a weighted space of holomorphic functions (see for instance \cite{Hitrik_Sjostrand_Two_Minicourses}). We use the notation $\mathcal{L}(\dd x) = \dd \jvRe x \, \dd \jvIm x$ for Lebesgue measure on $\Bbb{C}^n$. With $\Phi : \Bbb{C}^n \to \Bbb{R}$ we define the weighted $L^2$ norm
\begin{equation}
	\|f\|_{H_\Phi} = \left(\int_{\Bbb{C}^n} |f(x)|^2 \ee^{-4\pi \Phi(x)}\,\mathcal{L}(\dd x)\right)^{1/2}
\end{equation}
and the weighted space of holomorphic functions
\begin{equation}
	H_{\Phi} = \left\{f \in \operatorname{Hol}(\Bbb{C}^n) \::\: \|f\|_{H_\Phi} < \infty\right\}.
\end{equation}

The weights associated with quadratic phases are real-valued and real-quadratic, which means that $\Phi:\Bbb{C}^n \to \Bbb{R}$ can be written in the form
\begin{equation}\label{eq:Phi_LP}
	\Phi(x) = \frac{1}{2}Lx \cdot \ol{x} + \frac{1}{2}\jvRe(Px \cdot x).
\end{equation}
Here $L = 2\Phi''_{\ol{x}x}$ is a Hermitian matrix (called the Levi matrix) and $P = 2\Phi''_{xx}$ is symmetric (and gives the pluri-harmonic part of $\Phi$). The weights we consider are strictly plurisubharmonic, which in this context is equivalent to requiring that $L$ is a positive definite Hermitian matrix.

It comes as no surprise that the embedding
\begin{equation}\label{eq:def_iota_intro}
	\iota : H_{\Phi_1} \ni f \mapsto f \in H_{\Phi_2}
\end{equation}
between two quadratically weighted spaces is bounded if and only if $\Phi_2 \geq \Phi_1$. The precise value of the operator norm, however, is far from evident due to the assumption that the functions considered are holomorphic. These functions cannot concentrate arbitrarily close to the origin, and the norm ratio therefore represents a type of uncertainty principle. This is because each unitary FBI--Bargmann transform mapping to an $H_\Phi$ space corresponds to a wave packet decomposition, and in fact finding the norm of the embedding when $\Phi_j(x) = \frac{1}{2}t_j |x|^2$ is equivalent to finding the bottom of the spectrum of the quantum harmonic oscillator. (This is well-known; see for instance \cite[Thm.~4.12]{Aleman_Viola_2018}.) While the question is elementary, the author has not been able to find an elementary solution (except, with some difficulty, in dimension one, \cite[Thm.~1.2]{Viola_2016}).

The question is also relevant to the study of FBI--Bargmann transforms and Fourier integral operators with complex quadratic phase, particularly when applied to non-self-adjoint operators. In 1961, V.\ Bargmann introducted and studied what we call here $H_{\Phi_0}$, $\Phi_0(x) = \frac{1}{2}|x|^2$ and the Bargmann transform $\tilde{\barg}_0$ (see \eqref{eq:barg0_unitary} below) \cite[Eqs.~(1.2),~(2.3)]{Bargmann_1961}. In the time since these operators have been used in many applications in partial differential equations (for example, \cite{Sjostrand_1974, Sjostrand_SAM, Hormander_1983}), have been studied as objects of intrinsic interest (for example, \cite{Howe_1988, Folland_1989, Hormander_1995}), and have become part of the standard toolbox in the study of partial differential equations (for example, \cite{Martinez_book, Zworski_2012, Hitrik_Sjostrand_Two_Minicourses}). Motivated in part by applications to non-self-adjoint operators including subelliptic operators in kinetic theory, these operators and techniques are still the subject of active research (for example, \cite{Herau_Sjostrand_Stolk_2005, Hitrik_Pravda-Starov_2009, Pravda-Starov_Rodino_Wahlberg_2018, Coburn_Hitrik_Sjostrand_2019, Alphonse_Bernier_2019, Ben-Said_Nier_Viola_2020, White_2021, Karaki_2022} among many others). The author's motivation in studying this problem is to better understand $H_\Phi$-spaces and their close links with the metaplectic semigroup (Definition \ref{def:metaplectic_semigroup}).

Viewing the weighted spaces $H_{\Phi_j}, j = 1, 2$ as being related by Fourier integral operators with complex quadratic phase and applying the calculus of these operators \cite{Hormander_1995} we are able to find the operator norm of the embedding \eqref{eq:def_iota_intro} using the spectral theory of underlying (complex) canonical transformations. In order to state the main result, we introduce the matrix
\begin{equation}\label{eq:def_IPhi}
	\Bff{A}_\Phi = \ol{\begin{pmatrix} \ii L & 0 \\ \ii P & 1 \end{pmatrix}^{-1}} 
	\begin{pmatrix} 0 & 1 \\ 1 & 0 \end{pmatrix}
	\begin{pmatrix} \ii L & 0 \\ \ii P & 1 \end{pmatrix}, \quad L = 2\Phi''_{\ol{x}x}, P = 2\Phi''_{xx}.
\end{equation}
(Note that the entries $1$ or $0$ refer to the identity or zero matrices of size $n$-by-$n$.) This matrix is associated to the adjoint of phase-space shifts on an $H_\Phi$ space and to positivity conditions on weighted spaces; see for example \cite[Sec.~1.2]{Hitrik_Sjostrand_Two_Minicourses}. Notice that $\ol{\Bff{A}_\Phi} = \Bff{A}_\Phi^{-1}$. In Section \ref{s:adjoints} below, we recall some well-known results on $\Bff{A}_\Phi$.

\begin{theorem}\label{thm:main}
For $j = 1, 2$ let $\Phi_j : \Bbb{C}^n \to \Bbb{R}$ be real-quadratic with $(\Phi_j)''_{\ol{x}x}$ positive definite. Let $\Bff{A}_{\Phi_j}$ be as in \eqref{eq:def_IPhi}.
Then the embedding
\begin{equation}\label{eq:def_iota_thm}
	\iota : H_{\Phi_1} \to H_{\Phi_2}, \quad \iota f = f
\end{equation}
is bounded if and only if $\Phi_2 \geq \Phi_1$ on $\Bbb{C}^n$. In this case, there exist $\mu_1, \dots, \mu_n \in (0, 1]$ such that, counting for multiplicity, 
\begin{equation}\label{eq:embedding_norm_HPhi}
	\operatorname{Spec} \Bff{A}_{\Phi_2}^{-1}\Bff{A}_{\Phi_1} = \{\mu_j\}_{j=1}^n \cup \{\mu_j^{-1}\}_{j=1}^n
\end{equation}
and the operator norm of $\iota$ is given by
\begin{equation}
	\|\iota\| = \left(\frac{\det (\Phi_1)''_{\ol{x}x}}{\det (\Phi_2)''_{\ol{x}x}}\prod_{j=1}^n \mu_j\right)^{1/4}.
\end{equation}
\end{theorem}

\begin{remark}
The main difference between Theorem \ref{thm:main} and \cite[Thm.~1.3]{Viola_2017} (reproduced in Theorem \ref{thm:metaplectic_sg_norm} below) is the factor $(\det (\Phi_1)''_{\ol{x}x} / \det (\Phi_2)''_{\ol{x}x})^{1/4}$. This is essentially due to the fact that $\Bbb{C}^n$ has real dimension $2n$, which one sees for instance in Proposition \ref{prop:meta_mappings} below.
\end{remark}

\begin{remark}
Another way of writing \eqref{eq:embedding_norm_HPhi} is to define the stable subspace
\begin{equation}\label{eq:def_Es}
	\begin{aligned}
	E_s 
	&=
	\bigoplus_{j \::\: \mu_j \in (0, 1)} \ker (\Bff{A}_{\Phi_2}^{-1}\Bff{A}_{\Phi_1} - \mu_j)
	\\ &= 
	\{X \in \Bbb{C}^{2n}\::\: \exists C > 0, \forall N \in \Bbb{N}, \|(\Bff{A}_{\Phi_2}^{-1}\Bff{A}_{\Phi_1})^N X \| \leq C\ee^{-N/C}\}.
	\end{aligned}
\end{equation}
So long as $\Phi_2 \geq \Phi_1$, the product of eigenvalues in \eqref{eq:embedding_norm_HPhi} is equal to the determinant of $\Bff{A}_{\Phi_2}^{-1}\Bff{A}_{\Phi_1}$ restricted to $E_s$, giving
\begin{equation}
	\|\iota\| = \left(\frac{\det (\Phi_1)''_{\ol{x}x}}{\det (\Phi_2)''_{\ol{x}x}} \det(\Bff{A}_{\Phi_2}^{-1}\Bff{A}_{\Phi_1}|_{E_s})\right)^{1/4}
\end{equation}
\end{remark}

When we have a strict inequality $\Phi_1(x) < \Phi_2(x)$ for all $x \neq 0$, one can identify the centered Gaussian witnessing the maximum of $\|f\|_{H_{\Phi_2}} / \|f\|_{H_{\Phi_1}}$ again using $\Bff{A}_{\Phi_1}$ and $\Bff{A}_{\Phi_2}$. The formula is a straightforward modification of the formula for the ground state of the Weyl quantization of a positive definite quadratic form (see for instance \cite[Thm.~3.5]{Sjostrand_1974}). If the inequality $\Phi_1(x) \leq \Phi_2(x)$ is not in general strict, one can find a sequence of Gaussians maximizing $\|f\|_{H_{\Phi_{2, \eps}}} / \|f\|_{H_{\Phi_1}}$ when $\Phi_{2, \eps}(x) = \Phi_2(x) + \frac{1}{2}\eps|x|^2$. Because $\|f\|_{H_{\Phi_{2, \eps}}} \to \|f\|_{H_{\Phi_2}}$ as $\eps \to 0^+$ for $f \in H_{\Phi_2} \subseteq H_{\Phi_1}$ by the monotone convergence theorem, the limit of the norm ratios converges to the norm of the embedding $\iota$. The sequence of Gaussians will not generally converge to an integrable Gaussian, however. This is because the case of non-strict positivity can include operators like multiplication on $L^2(\Bbb{R})$ by $\ee^{-x^2/2}$, whose operator norm of $1$ is not attained for any $L^2(\Bbb{R})$-function.

\begin{theorem}\label{thm:norm_where}
For $j = 1, 2$ let $\Phi_j : \Bbb{C}^n \to \Bbb{R}$ be real-quadratic with $(\Phi_j)''_{\ol{x}x}$ positive definite. Let $\Bff{A}_{\Phi_j}$ be as in \eqref{eq:def_IPhi}. Suppose furthermore that $\Phi_1(x) < \Phi_2(x)$ for all $x \in \Bbb{C}^n\backslash\{0\}$. Then $\operatorname{Spec}\Bff{A}_{\Phi_2}^{-1}\Bff{A}_{\Phi_1} \subseteq (0, 1) \cup (1, +\infty)$ and there exists $T$, a symmetric $n$-by-$n$ matrix with complex entries, such that the stable subspace \eqref{eq:def_Es} of $\Bff{A}_{\Phi_2}^{-1}\Bff{A}_{\Phi_1}$ is the graph of $T$, meaning
\begin{equation}
	\{(x, Tx)\}_{x \in \Bbb{C}^n} = \bigoplus_{j \::\: \mu_j \in (0, 1)} \ker(\Bff{A}_{\Phi_2}^{-1}\Bff{A}_{\Phi_1} - \mu_j).
\end{equation}
With $\iota$ the embedding from \eqref{eq:def_iota_thm}, the Gaussian $g_T(x) = \exp(\pi\ii Tx\cdot x)$ witnesses the norm of $\iota$:
\begin{equation}
	\|\iota\| = \frac{\|g_T\|_{H_{\Phi_2}}}{\|g_T\|_{H_{\Phi_1}}}.
\end{equation}
\end{theorem}

\begin{remark}
Using \cite{Viola_2017}, it is straightforward to extend these results to $\Phi_1, \Phi_2$ real-valued polynomials of degree $2$ so long as $(\Phi_j)''_{\ol{x}x}$ are positive definite and so long as the quadratic part of $\Phi_2$ is strictly larger than the quadratic part of $\Phi_1$. To make this latter condition concrete, one would assume that
\begin{equation}
	\liminf_{|x|\to +\infty} |x|^{-2}(\Phi_2(x) - \Phi_1(x)) > 0.
\end{equation}
An extension to more general $\Phi_j$ would certainly be interesting, but the author does not see how to obtain similarly sharp results in a broader setting.
\end{remark}

\begin{planofpaper*}
In Section \ref{s:examples} which follows, we apply the main result of the paper to two examples. Sections \ref{s:metaplectic} and \ref{s:adjoints} summarize generally well-known results: Section \ref{s:metaplectic} concerns the metaplectic group and FBI--Bargmann transforms, while Section \ref{s:adjoints} concerns adjoints of phase-space shifts on $H_\Phi$ spaces. Section \ref{s:proof_norm} contains the proof of Theorem \ref{thm:main}, while Section \ref{s:proof_where} contains the proof of Theorem \ref{thm:norm_where}.
\end{planofpaper*}

\begin{acknowledgements*}
The author would like to thank Johannes Sjöstrand and Michael Hitrik for helpful discussions and comments.
\end{acknowledgements*}

\begin{notation*}
To make certain formulas more readable, we use 
\begin{equation}
	e(\theta) = \ee^{2\pi\ii \theta}
\end{equation}
throughout, and we reiterate that
\begin{equation}
	\mathcal{L}(\dd x) = \dd(\jvRe x) \wedge \dd(\jvIm x) = (-2\ii)^{-n}\dd x \wedge \dd \ol{x}
\end{equation}
is Lebesgue measure on $\Bbb{C}^n$. We also recall that a real-valued quadratic form $\Phi:\Bbb{C}^n \to \Bbb{R}$ is strictly plurisubharmonic if and only if the Hermitian matrix $\partial_{\ol{x}} \partial_x \Phi > 0$ in the sense of positive definite matrices.

We use $\Bbb{M}_n(\Bbb{C})$ for the space of $n$-by-$n$ complex matrices. A matrix plus a scalar (such as $A + \ii$) indicates the matrix plus the scalar times the corresponding identity matrix; the same principle applies to scalars in block matrices like \eqref{eq:def_skewmat}. Bold-faced names for matrices (like $\Bff{A}_\Phi$) are reserved for canonical transformations. We write $x \cdot y = \sum_{j = 1}^n x_j y_j$ for the dot product on $\Bbb{R}^n$ or $\Bbb{C}^n$; the notation $\langle f, g\rangle$ refers to a sesquilinear inner product. In particular, on $H_\Phi$ spaces,
\begin{equation}
	\langle f, g\rangle_{H_\Phi} = \int_{\Bbb{C}^n} f(x)\ol{g(x)}e(2\ii\Phi(x))\,\mathcal{L}(\dd x).
\end{equation}
\end{notation*}

\section{Examples}\label{s:examples}

To illustrate the main results of this article, we apply Theorem \ref{thm:main} to two simple examples treated in previous works.

\begin{example}
If the pluriharmonic parts of $\Phi_1$ and $\Phi_2$ vanish, meaning $\Phi_j(x) = \frac{1}{2}L_j x\cdot \ol{x}$ for $L_j$ Hermitian, then
\begin{equation}
	\Bff{A}_{\Phi_2}^{-1} \Bff{A}_{\Phi_1} = \ol{\begin{pmatrix} 0 & \ii \ol{L_2^{-1}} \\ \ii L_2 & 0 \end{pmatrix}} \begin{pmatrix} 0 & \ii \ol{L_1^{-1}} \\ \ii L_1 & 0 \end{pmatrix} = \begin{pmatrix} L_2^{-1} L_1 & 0 \\ 0 & \ol{L_2 L_1^{-1}}\end{pmatrix}.
\end{equation}
By a change of variables, $\Phi_2 \geq \Phi_1$ if and only if $L_1^{-1/2}L_2 L_1^{-1/2} \geq 1$ in the sense of positive semi-definite matrices, which we assume in what follows. In this case,
\begin{equation}
	\operatorname{Spec}(L_1^{-1/2}L_2 L_1^{-1/2}) = \operatorname{Spec}(L_1^{-1} L_2) \in [1, +\infty).
\end{equation}
The spectrum of $\Bff{A}_{\Phi_2}^{-1}\Bff{A}_{\Phi_1}$ lying in $(0, 1]$ is therefore precisely the spectrum of $L_2^{-1}L_1$, and the product of these eigenvalues (counted for multiplicity) is $\det(L_2^{-1}L_1)$. By Theorem \ref{thm:main},
\begin{equation}
	\|\iota\| 
	=
	\left(\frac{\det L_1}{\det L_2} \det(L_2^{-1}L_1)\right)^{1/4} 
	= 
	\sqrt{\frac{\det L_1}{\det L_2}}.
\end{equation}
This reproduces \cite[Cor.~4.13]{Aleman_Viola_2018}, which used a similar approach to the present work. Specifically, embeddings between weights with no pluriharmonic part are shown to be unitarily equivalent to certain quantum harmonic oscillators.

If we assume the strict inequality $\Phi_2 > \Phi_1$ on $\Bbb{C}^n \backslash \{0\}$, then $\operatorname{Spec} L_2^{-1}L_1 \subseteq (0, 1)$. Every eigenvector $v$ of $L_2^{-1}L_1$ corresponds to an eigenvector $(v, 0)$ of $\Bff{A}_{\Phi_2}^{-1} \Bff{A}_{\Phi_1}$. Therefore in Theorem \ref{thm:norm_where} we have $T = 0$ and the maximum of the norm of $\iota$ occurs on the constant functions.
\end{example}

\begin{example}
In dimension one, let us write
\begin{equation}
	\Phi_2(x) = \frac{1}{2}a|x|^2 + \jvRe(bx^2), \quad a > 0, b \in \Bbb{C}
\end{equation}
and $\Phi_1(x) = \frac{1}{2}|x|^2$ (which, as detailed in Example \ref{ex:Phi_reduction} below, can be obtained by some simple transformations). Because the argument of $bx^2$ may be changed by replacing $x$ with $e(\theta)x$ for varying $\theta \in \Bbb{R}$, we see that $\Phi_2 \geq \Phi_1$ if and only if $a - |b| \geq 1$. 

Using that $\Bff{A}_\Phi^{-1} = \ol{\Bff{A}_\Phi}$,
\begin{equation}
	\Bff{A}_{\Phi_2}^{-1}\Bff{A}_{\Phi_1} = \begin{pmatrix} -\ol{b}/a & -\ii/a \\ -\ii(a - |b|^2/a) & -b/a\end{pmatrix} \begin{pmatrix} 0 & \ii \\ \ii & 0\end{pmatrix} = \frac{1}{a}\begin{pmatrix} 1 & -\ii \ol{b} \\ -\ii b & a^2 - |b|^2 \end{pmatrix}.
\end{equation}
Computing the eigenvalues and inserting into Theorem \ref{thm:main} (where $L_2 = a$ and $L_1 = 1$) gives, when $a - |b| \geq 1$,
\begin{equation}
	\|\iota\| = \left(\frac{1}{2a^2}\left(1 + a^2 - |b|^2 - \sqrt{(1 + a^2 - |b|^2)^2 - 4a^2}\right)\right)^{1/4}.
\end{equation}
One can check that this agrees with the formula in \cite[Thm.~1.2]{Viola_2016}, which was obtained through an elementary calculus argument.

A routine computation reveals that, if $a - |b| > 1$ and if
\begin{equation}
	\tau = -\frac{1}{2\ol{b}}\ii\left(1 - a^2 + |b|^2 + \sqrt{(1 + a^2 - |b|^2)^2 - 4a^2}\right),
\end{equation}
then $(1, \tau)$ is an eigenvector of $\Bff{A}_{\Phi_2}^{-1}\Bff{A}_{\Phi_1}$ with eigenvalue in $(0, 1)$. Therefore it is
\begin{equation}
	g_\tau(x) = \ee^{\pi\ii \tau x^2}
\end{equation}
which maximizes the norm ratio for the embedding $\iota:H_{\Phi_1} \to H_{\Phi_2}$.

In the limiting case $a - |b| = 1$ (assuming $a > 1$), one computes
\begin{equation}
	\tau = \ii \frac{b}{|b|}.
\end{equation}
Since $|\tau| = 1$, $g_\tau \notin H_{\Phi_1}$ since $\Phi_1 = \frac{1}{2}|x|^2$. But one can compute that for $\delta \in [0, 1)$,
\begin{equation}
	\frac{\|g_{\delta \tau}\|_{H_{\Phi_2}}}{\|g_{\delta\tau}\|_{H_{\Phi_1}}} = \left(\frac{4(2a +  \delta - 1)(1-\delta)}{4(1-\delta^2)}\right)^{-1/4} \to a^{-1/4}, \quad \delta \to 1^-,
\end{equation}
and $a^{-1/4}$ is indeed the norm of the embedding from $H_{\Phi_1}$ into $H_{\Phi_2}$ in the case $a - |b| = 1$.
\end{example}

\section{Some metaplectic operators on weighted spaces}\label{s:metaplectic}

The proof of the Theorem is a relatively simple consequence of the theory of the metaplectic semigroup \cite{Hormander_1983, Howe_1988, Folland_1989, Hormander_1995}. For the unfamiliar, we try to describe this theory as concisely as possible. 

\subsection{Phase-space shifts and the metaplectic semigroup}

For $X = (x, \xi), Y = (y, \eta) \in \Bbb{C}^{2n}$, define the symplectic form
\begin{equation}
	\sigma((x, \xi), (y, \eta)) = \xi \cdot y - \eta \cdot x,
\end{equation}
where $\xi \cdot y = \sum_{j = 1}^n \xi_j y_j$ is the usual (non-Hermitian) scalar product on $\Bbb{C}^n$. Let $\shift_Y$ denote a shift in phase space defined by (where $e(\theta) = \ee^{2\pi\ii\theta}$)
\begin{equation}\label{eq:def_shift}
	\shift_{(y, \eta)} f(x) = e(-\frac{1}{2}y\cdot \eta + \eta x)f(x-y).
\end{equation}
When $Y \in \Bbb{R}^{2n}$ this operator is unitary on $L^2(\Bbb{R}^n)$; if $Y \in \Bbb{C}^{2n} \backslash \Bbb{R}^{2n}$ one can define $\shift_Y$ as an unbounded operator defined on a core of wave packets $\{\shift_X g_T\}_{X \in \Bbb{R}^{2n}}$ with $g_T$ an integrable Gaussian. Note that, when $D_x = (2\pi\ii)^{-1}\nabla_x$,
\begin{equation}
	\shift_Y = e(\sigma((y, \eta), (x, D_x))) = \exp(2\pi\ii \eta \cdot x - y\cdot \nabla_x)
\end{equation}
in the sense of an evolution equation on functions. Shifts compose via the rule
\begin{equation}
	\shift_X\shift_Y = e(\frac{1}{2}\sigma(X, Y))\shift_{X+Y},
\end{equation} 
making $\{e(\theta)\shift_Y\::\:\theta \in \Bbb{R}, Y \in \Bbb{R}^{2n}\}$ the Schr\"odinger representation of the Heisenberg group (see for example \cite[Ch.~1,~\S 3]{Folland_1989}).

A linear map $\Bff{M} : \Bbb{C}^{2n} \to \Bbb{C}^{2n}$ is said to be canonical if it preserves $\sigma$, meaning $\sigma(\Bff{M}X, \Bff{M}Y) = \sigma(X, Y)$ for all $X, Y \in \Bbb{C}^{2n}$. A linear canonical transformation is said to be positive if
\begin{equation}
	-\ii (\sigma(\Bff{M}X, \ol{\Bff{M}X}) - \sigma(X, \ol{X})) \geq 0, \quad \forall X \in \Bbb{C}^{2n}.
\end{equation}

We are now able to define the metaplectic semigroup following \cite{Hormander_1995}. In defining the operators in the metaplectic semigroup, we recall that when $g_S(x) = e(\frac{1}{2}Sx\cdot x) = \ee^{\pi\ii Sx\cdot x}$ for $S \in \Bbb{M}_n(\Bbb{C})$ symmetric with $\jvIm S$ positive definite, the family of shifted Gaussians $\{\shift_X g_S\}_{X \in \Bbb{R}^{2n}}$ has dense span in $L^2(\Bbb{R}^n)$. A bounded operator on $L^2(\Bbb{R}^n)$ can therefore be defined via its action on this family. 

\begin{definition}[The metaplectic semigroup]\label{def:metaplectic_semigroup}
Let
\begin{equation}
	\Bff{M} = \begin{pmatrix} A & B \\ C & D\end{pmatrix},
\end{equation} 
where $A, B, C, D \in \Bbb{M}_n(\Bbb{C})$, be a positive complex linear canonical transformation. An element of the metaplectic semigroup quantizing $\Bff{M}$ is a bounded operator $\mathcal{M}: L^2(\Bbb{R}^n) \to L^2(\Bbb{R}^n)$ satisfying
\begin{enumerate}[(i)]
\item the Egorov relation
\begin{equation}\label{eq:Egorov}
	\mathcal{M}\shift_Y = \shift_{\Bff{M}Y}, \mathcal{M}, \quad \forall Y \in \Bbb{C}^{2n},
\end{equation}
and
\item for every $T \in \Bbb{M}_n(\Bbb{C})$ symmetric with $\jvIm T$ positive definite, there exists a choice of sign $\eps(T) \in \{\pm 1\}$ such that, writing $g_S(x) = e(\frac{1}{2}Sx\cdot x)$,
\begin{equation}\label{eq:meta_gaussian}
	\mathcal{M}g_T = \eps(T) \det(A + BT)^{-1/2} g_{T'}, \quad T' = (C+DT)(A+BT)^{-1}.
\end{equation}
\end{enumerate}
\end{definition}

As a special case of \eqref{eq:Egorov} and \eqref{eq:meta_gaussian} one has the action of $\mathcal{M}$ on any wave packet $\shift_X g_S$ for $X \in \Bbb{R}^{2n}$. Once one has chosen $\eps(T)$ in \eqref{eq:meta_gaussian} for a given $T$, there is a unique choice of $\eps(S)$ for every other $S$ symmetric with positive definite imaginary part; it suffices to require that $\det(A + BS)^{1/2}\mathcal{M}g_S(0) \in \{\pm 1\}$ be a continuous function of $S$. There are therefore exactly two operators in the metaplectic semigroup quantizing any given positive complex linear canonical transformation $\Bff{M}$, each corresponding to a choice of sign. 

The metaplectic semigroup extends the metaplectic group, which is the subset of the metaplectic semigroup quantizing real canonical transformations. (Equivalently, the metaplectic group is the subset of the metaplectic semigroup consisting of unitary operators on $L^2(\Bbb{R}^n)$.)

The metaplectic semigroup is closed under composition: if $\mathcal{M}_1$ and $\mathcal{M}_2$ are elements of the metaplectic semigroup quantizing $\Bff{M}_1$ and $\Bff{M}_2$, then $\mathcal{M}_1\mathcal{M}_2$ is an element of the metaplectic semigroup quantizing $\Bff{M}_1 \Bff{M}_2$. (It is a straightforward exercise to verify \eqref{eq:Egorov} and \eqref{eq:meta_gaussian} for the composition, modulo an argument from positivity that the matrix $A + BT$ remains invertible in \eqref{eq:meta_gaussian}.)

One may also characterize the metaplectic semigroup via its generators which may be taken to be changes of variables \eqref{eq:def_scaleop}, multiplication by Gaussians \eqref{eq:def_skewop} where the phase has positive semi-definite imaginary part, and exponentials of the quantum harmonic oscillator $\exp(-\pi t(x_1^2 + D_{x_1}^2))$ where $D_x = (2\pi\ii)^{-1}\nabla_x$ and $\jvRe t \geq 0$. (We remark that this last family includes a partial Fourier transform when $t = \pi\ii/2$.)

\subsection{The action of some operators on weighted spaces}

Some metaplectic operators which allow us to pass from one weight to another are given by multiplication by Gaussians and changes of variables. For $T$ a symmetric matrix, let
\begin{equation}\label{eq:def_skewop}
	\skewop_{T} f(x) = e(\frac{1}{2}x\cdot Tx)f(x)
\end{equation}
be the operator quantizing
\begin{equation}\label{eq:def_skewmat}
	\Bff{W}_{T} = \begin{pmatrix} 1 & 0 \\ T & 1\end{pmatrix}.
\end{equation}
For $G \in \Bbb{M}_n(\Bbb{C})$ invertible, let
\begin{equation}\label{eq:def_scaleop}
	\scaleop_G f(x) = (\det G)^{1/2} f(Gx).
\end{equation}
(In general, some sign considerations arise in the square root, but our application will involve $G$ positive definite Hermitian where the usual square root on $(0, \infty)$ can be used.) The operator $\scaleop_G$ quantizes
\begin{equation}\label{eq:def_scalemat}
	\Bff{V}_G = \begin{pmatrix} G^{-1} & 0 \\ 0 & G^\top\end{pmatrix}.
\end{equation}
We remark that it is straightfoward to verify \eqref{eq:Egorov} and \eqref{eq:meta_gaussian} from the definitions of $\skewop_T$ and $\scaleop_G$.

Writing out the definitions of the relevant norms and a change of variables immediately gives the following description of the spaces on which $\skewop_{\ii P}$ and $\scaleop_G$ act; we also include the action of a phase-space shift on $H_\Phi$-spaces. We emphasize that the factor in front of $\scaleop_G$ which gives a \emph{metaplectic} operator is different from the factor in front of $\tilde{\scaleop}_G$ which gives a \emph{unitary} operator between $H_\Phi$ spaces because the number of (real) variables for an $H_\Phi$ space is $2n$.

\begin{proposition}\label{prop:meta_mappings}
Let $\Phi:\Bbb{C}^n \to \Bbb{R}$ be real-quadratic with $\Phi''_{\ol{x}x}$ positive definite. Let $G, T \in \Bbb{M}_n(\Bbb{C})$ with $\det G \neq 0$ and $T^\top = T$, and let $Y = (y, \eta) \in \Bbb{C}^{2n}$. Let
\begin{equation}
	\Phi_T(x) = \Phi(x) + \frac{1}{2}\jvRe (x\cdot (\ii T)x),
\end{equation}
let
\begin{equation}
	\Phi_G(x) = \Phi(Gx),
\end{equation}
and let
\begin{equation}\label{eq:Phi_shifted}
	\Phi_Y(x) = \Phi(x-y) + \jvIm(\frac{1}{2}y\cdot\eta - \eta \cdot x)
\end{equation}
Recall $\skewop_{\ii P}$, $\scaleop_{G}$, and $\shift_Y$ from \eqref{eq:def_skewop}, \eqref{eq:def_scaleop}, and \eqref{eq:def_shift}. Then
\begin{equation}
	\skewop_{T} : H_{\Phi} \to H_{\Phi_T},
\end{equation}
\begin{equation}\label{eq:def_scaleop_HPhi}
	\tilde{\scaleop}_G = (\det \ol{G})^{1/2}\scaleop_G : H_{\Phi} \to H_{\Phi_G},
\end{equation}
and
\begin{equation}
	\shift_Y : H_{\Phi} \to H_{\Phi_Y}
\end{equation}
are unitary.
\end{proposition}

\begin{proof}
In each case, we assume that the target weight is some 
\begin{equation}
	\Phi_1(x) = \frac{1}{2}(L_1 x\cdot \ol{x} + \jvRe (P_1 x\cdot x))
\end{equation}
to be determined. For the operator $\mathcal{K}$ in question we then compute $\langle \mathcal{K} f, \mathcal{K} g\rangle_{H_{\Phi_1}}$ and we find $\Phi_1$ such that $\langle \mathcal{K} f, \mathcal{K} g\rangle_{H_{\Phi_1}} = \langle f, g\rangle_{H_{\Phi}}$.

First
\begin{equation}
	\begin{aligned}
	\langle \skewop_T  f, \skewop_T g\rangle_{H_{\Phi_1}} 
	&= 
	\int_{\Bbb{C}^n} e(\frac{1}{2}Tx\cdot x)f(x)\ol{e(\frac{1}{2}Tx\cdot x)g(x)}e(2\ii\Phi_1(x))\,\mathcal{L}(\dd x)
	\\ &=
	\int_{\Bbb{C}^n} f(x)\ol{g(x)}e(2\ii(\Phi_1(x) - \frac{1}{4}\ii Tx \cdot x + \frac{1}{4}\ii \ol{Tx\cdot x})\,\mathcal{L}(\dd x).
	\end{aligned}
\end{equation}
We obtain $\Phi_T$ from the observation
\begin{equation}
	\Phi(x) = \Phi_1(x) - \frac{1}{4}\ii Tx \cdot x + \frac{1}{4}\ii \ol{Tx\cdot x} \iff \Phi_1(x) = \Phi_T(x).
\end{equation}

As for $\tilde{\mathcal{V}}_G$, we make a change of variables $x' = Gx$, observing that $\mathcal{L}(\dd x) = |\det G|^{-2}\mathcal{L}(\dd x')$ because $x, x' \in \Bbb{C}^n \sim \Bbb{R}^{2n}$:
\begin{equation}
	\begin{aligned}
	\langle \tilde{\scaleop}_G f, \tilde{\scaleop}_G g\rangle_{H_{\Phi_1}}
	&=
	\int_{\Bbb{C}^n} |\det G|^2 f(Gx) \ol{g(Gx)} e(2\ii \Phi_1(x))\,\mathcal{L}(\dd x)
	\\ &=
	\int_{\Bbb{C}^n} f(x')\ol{g(x')}e(2\ii \Phi_1(G^{-1}x'))\,\mathcal{L}(\dd x').
	\end{aligned}
\end{equation}
We obtain $\Phi_G$ from the observation 
\begin{equation}
	\Phi_1(G^{-1}x) = \Phi(x) \iff \Phi_1(x) = \Phi_G(x).
\end{equation}

Finally, a similar exercise for $\shift_Y$ (with the change of variables $x' = x-y$) gives
\begin{equation}
	\begin{aligned}
	\langle \shift_Y & f, \shift_Y g\rangle_{H_{\Phi_1}}
	\\ &=
	\int_{\Bbb{C}^n} e(-\frac{1}{2}y\cdot\eta + \eta \cdot x)f(x-y)\ol{e(-\frac{1}{2}y \cdot \eta + \eta \cdot x)g(x-y)}e(2\ii\Phi_1(x))\,\mathcal{L}(\dd x)
	\\ &=
	\int_{\Bbb{C}^n} f(x')\overline{g(x')}e(2\ii(\jvIm(\frac{1}{2}y\cdot \eta + \eta \cdot x') + \Phi_1(x' + y)))\,\mathcal{L}(\dd x').
	\end{aligned}
\end{equation}
As in the previous two cases, we obtain $\Phi_Y$ by checking that
\begin{equation}
	\jvIm(\frac{1}{2}y\cdot \eta + \eta \cdot x) + \Phi_1(x + y) = \Phi(x) \iff \Phi_1(x) = \Phi_Y(x).
\end{equation}
\end{proof}

\begin{corollary}\label{cor:Lambda_Phi}
Let $\Phi:\Bbb{C}^n \to \Bbb{R}$ be real-quadratic with $\Phi''_{\ol{x}x}$ positive definite. For any $Y = (y, \eta) \in \Bbb{C}^{2n}$ let $\Phi_Y$ be as in \eqref{eq:Phi_shifted}. Then $\Phi_Y = \Phi$ if and only if $Y$ is of the form
\begin{equation}
	Y = (y, -2\ii\Phi'_x(y)), \quad y \in \Bbb{C}^n.
\end{equation}
\end{corollary}

\begin{proof}
Let $L = 2\Phi''_{\ol{x}x}$ and $P = 2\Phi''_{xx}$, so
\begin{equation}
	-2\ii\Phi'_x(y) = -\ii(Py + \ol{Ly}).
\end{equation}
Expanding out $\Phi_Y$ gives
\begin{equation}
	\Phi_Y(x) = \Phi(x) - \frac{1}{2}(x - \frac{1}{2}y)\cdot(\ol{Ly} + Py - \ii \eta) - \frac{1}{2}(\ol{x} - \frac{1}{2}\ol{y})\cdot(Ly + \ol{Py} + \ii \ol{\eta}).
\end{equation}
The claim in the corollary is then obvious.
\end{proof}

\subsection{Some FBI--Bargmann transforms}\label{ssec:Bargmann}

Finally, we recall a standard FBI--Bargmann transform
\begin{equation}\label{eq:barg0_unitary}
	\tilde{\barg}_0 f(x) = 2^{3n/4}\int_{\Bbb{C}^n} e(\frac{1}{2}\ii x^2 - \ii\sqrt{2}xy + \frac{1}{2}\ii y^2)f(y)\,\dd y,
\end{equation}
which is a unitary map
\begin{equation}
	\tilde{\barg}_0: L^2(\Bbb{R}^n) \to H_{\Phi_0}, \quad \Phi_0(x) = \frac{1}{2}|x|^2.
\end{equation}
The FBI--Bargmann transform $\tilde{\barg}_0$ quantizes, as in \eqref{eq:Egorov}, the complex canonical transformation
\begin{equation}
	\Bff{B}_0 = \frac{1}{\sqrt{2}}\begin{pmatrix} 1 & -\ii \\ -\ii & 1\end{pmatrix}.
\end{equation}

The inclusion of a tilde above the $\barg$ is to emphasize that $\tilde{\barg}_0$ is not metaplectic in the sense of \eqref{eq:meta_gaussian}, similarly to \eqref{eq:def_scaleop_HPhi}. Indeed, we check that when $g_T(x) = e(\frac{1}{2}Tx\cdot x)$ for $T$ symmetric with $\jvIm T > 0$, writing $w = \ii\sqrt{2}(\ii + T)^{-1}x$,
\begin{equation}\label{eq:barg_gT}
	\begin{aligned}
	\tilde{\barg}_0 g_T(x)
	&= 
	2^{3n/4}\int_{\Bbb{C}^n} e(\frac{1}{2}\ii x^2 + \frac{1}{2}(\ii + T)(y - w)\cdot(y-w) + (\ii + T)^{-1}x\cdot x)\,\dd y
	\\ &=
	2^{3n/4}\det(\ii + T)^{-1/2} e(\frac{1}{2}(1 + \ii T)(\ii + T)^{-1}x\cdot x)
	\end{aligned}
\end{equation}
When $A = D = \frac{1}{\sqrt{2}}$ and $B = C = -\frac{1}{\sqrt{2}}\ii$, it is true that $T'$ in \eqref{eq:meta_gaussian} is $(1 + \ii T)(\ii + T)^{-1}$. On the other hand, $\det(A + BT)^{-1/2} = 2^{n/4}\det(\ii + T)^{-1/2}$. It is therefore
\begin{equation}
	\barg_0 = 2^{-n/2}\tilde{\barg}_0
\end{equation}
which respects the metaplectic rule \eqref{eq:meta_gaussian}.

Given any $\Phi:\Bbb{C}^n \to \Bbb{R}$ real-quadratic with $\Phi''_{\ol{x}x}$ positive definite, we can construct a unitary FBI--Bargmann transform taking $L^2(\Bbb{R}^n)$ to $H_\Phi(\Bbb{C}^n)$. If $\Phi$ is in the form \eqref{eq:Phi_LP}, then
\begin{equation}\label{eq:barg_unitary}
	\tilde{\barg} = \skewop_{-\ii P} \tilde{\scaleop}_{L^{1/2}} \barg_0 : L^2(\Bbb{R}^n) \to H_\Phi(\Bbb{C}^n)
\end{equation}
is unitary by Proposition \ref{prop:meta_mappings}. We emphasize again that $\tilde{\barg}$ does not satisfy \eqref{eq:meta_gaussian}, and it is instead (recalling that $L$ is positive definite)
\begin{equation}\label{eq:barg_meta}
	\barg = \skewop_{-\ii P} \scaleop_{L^{1/2}} \barg_0 = 2^{-n/2}(\det L)^{-1/4}\tilde{\barg}
\end{equation}
which does so.

\begin{example}\label{ex:Phi_reduction}
To simplify computations, it is sometimes practical to reduce one weight to the standard weight $\Phi_0(x) = \frac{1}{2}|x|^2$. If $\Phi_j = \frac{1}{2}(L_j x\cdot \ol{x} + \jvRe(P_j x \cdot x))$ for $j = 1, 2$ with $L_j$ Hermitian positive definite and $P_j$ symmetric, then 
\begin{equation}
	\mathcal{U} = \tilde{\scaleop}_{L_1^{-1/2}} \skewop_{\ii P_1}: H_{\Phi_1} \to H_{\Phi_0}
\end{equation}
is unitary by Proposition \ref{prop:meta_mappings}. Similarly, when
\begin{equation}
	\Phi(x) = \frac{1}{2}\left(L_1^{-1/2}L_2 L_1^{-1/2}x \cdot \ol{x} + \jvRe( \ol{L}_1^{-1/2}(P_2 - P_1) L_1^{-1/2}x\cdot x)\right),
\end{equation}
the transformation
\begin{equation}
	\mathcal{U}: H_{\Phi_2} \to H_{\Phi}
\end{equation}
is also unitary. Therefore for any holomorphic function $f$, $f \in H_{\Phi_1}$ if and only if $g = \mathcal{U}f \in H_{\Phi_0}$ and
\begin{equation}
	\frac{\|f\|_{H_{\Phi_2}}}{\|f\|_{H_{\Phi_1}}} = \frac{\|g\|_{H_\Phi}}{\|g\|_{H_{\Phi_0}}}.
\end{equation}
Note that in both cases we are replacing $\Phi_j$ with
\begin{equation}
	\Phi_j(L_1^{-1/2}x) - \frac{1}{2}\jvRe  (\ol{L}_1^{-1/2}P_1 L_1^{-1/2}x\cdot x),
\end{equation}
so
\begin{equation}
	\Phi_2 \geq \Phi_1 \iff \Phi \geq \Phi_0.
\end{equation}
\end{example}

\section{Adjoints on $H_\Phi$ spaces}\label{s:adjoints}

We turn to the study of $\Bff{A}_\Phi$ from \eqref{eq:def_IPhi} from the point of view of adjoints of shift operators on $H_\Phi$-spaces. We remark that the decomposition in \eqref{eq:def_IPhi} is not necessarily the most practical for every situation. One could certainly multiply out to obtain
\begin{equation}\label{eq:IPhi_multiplied_out}
	\Bff{A}_\Phi = \begin{pmatrix} - \ol{L^{-1}} P & \ii \ol{L^{-1}} \\ \ii(L - \ol{PL^{-1}}P) & - \ol{P L^{-1}}\end{pmatrix}.
\end{equation}
Or one could maximize the use of the matrices in \eqref{eq:def_skewmat} and \eqref{eq:def_scalemat}: if $\Bff{R}(y, \eta) = \ii(\eta, y)$ (the factor of $\ii$ makes $\Bff{R}$ canonical), then
\begin{equation}
	\begin{aligned}
	\Bff{A}_\Phi 
	&= 
	\Bff{W}_{\ii \ol{P}} \Bff{R} \Bff{V}_L^{-1}  \Bff{W}_{\ii P}
	\\ &=
	\Bff{W}_{\ii \ol{P}} \Bff{V}_{\ol{L}} \Bff{R} \Bff{W}_{\ii P}.
	\end{aligned}
\end{equation}

\begin{proposition}\label{prop:adjoint_shifts}
Let $\Phi:\Bbb{C}^n \to \Bbb{R}$ be real-quadratic with $\Phi''_{\ol{x}x}$ positive definite. Let $\Bff{A}_\Phi$ be as in \eqref{eq:def_IPhi}. Then the adjoint of a phase-space shift \eqref{eq:def_shift} by $Y \in \Bbb{C}^{2n}$, as an operator on $H_\Phi$, is
\begin{equation}
	\shift_Y^* = \shift_{-\ol{\Bff{A}_\Phi Y}}.
\end{equation}
\end{proposition}

\begin{proof}
As usual, let $L = 2\Phi''_{\ol{x}x}$ and let $P = 2\Phi''_{xx}$. Let
\begin{equation}
	\Psi(x, \ol{y}) = \frac{1}{2}Lx\cdot \ol{y} + \frac{1}{4}Px\cdot x + \frac{1}{4}\ol{Py}\cdot\ol{y}
\end{equation}
so that $\Phi(x) = \Psi(x, \ol{x})$ and
\begin{equation}
	\langle f, g\rangle_\Phi = \int_{\Bbb{C}^n} f(x)\ol{g(x)}e(2\ii \Psi(x, \ol{x}))\,\mathcal{L}(\dd x).
\end{equation}
Since $\mathcal{L}(\dd x) = (-2\ii)^{-n}\dd x \wedge \dd \overline{x}$, we can analyze the integral giving the $H_\Phi$ inner product in holomorphic and anti-holomorphic coordinates. Formally, changes of variables are accomplished via contour deformation on dense sets of functions corresponding to rapidly decaying integrals. 

As an example, let $\Phi(x) = \Phi_0(x) = \frac{1}{2}|x|^2$ and $\Psi_0(x, \ol{y}) = \frac{1}{2}x\cdot \ol{y}$. Fix $y \in \Bbb{C}^n$ and let $f, g$ be polynomials; we define $\ol{g}$ by taking the complex conjugate of the coefficients of $g$ so that $\ol{g(z)} = \ol{g}(\ol{z})$. When $x = x_1 + \ii x_2$ for $(x_1, x_2) \in \Bbb{R}^{2n}$, we may use a contour deformation to make the change of variables $(z_1, z_2) = (x_1 - y/2, x_2 + \ii y/2)$. We therefore have
\begin{equation}
\begin{aligned}
	\int_{\Bbb{C}^n} & f(x-y)\ol{g(x)}e(2\ii \Psi_0(x, \ol{x}))\,\mathcal{L}(\dd x)
	\\ &=
	\iint_{\Bbb{R}^{2n}} f(x_1 + \frac{1}{2}y + \ii(x_2 - \frac{1}{2}\ii y))\ol{g}(x_1 - \ii x_2)\ee^{-2\pi(x_1^2 + x_2^2)}\,\dd x_1\,\dd x_2
	\\ &=
	\iint_{\Bbb{R}^{2n}} f(z_1 + \ii z_2)\ol{g}(z_1 + \frac{1}{2}y - \ii(z_2 - \frac{1}{2}\ii y))\ee^{-2\pi((z_1 + \frac{1}{2}y)^2 + (z_2 - \frac{1}{2}\ii y)^2)}\,\dd z_1 \,\dd z_2
	\\ &=
	\iint_{\Bbb{R}^{2n}} f(z_1 + \ii z_2)\ol{g}(z_1 - \ii z_2)\ee^{-2\pi(z_1^2 + z_2^2 + (z_1 - \ii x_2)\cdot y)}\,\dd x_1 \,\dd x_2
	\\ &=
	\int_{\Bbb{C}^n} f(x)\ol{g(x)}e(2\ii\Psi_0(z+y, \ol{z})).
\end{aligned}
\end{equation}

We return to an arbitrary shift on an $H_\Phi$ space, applying this type of change of variables. If $Y = (y, \eta)$,
\begin{equation}
	\begin{aligned}
	\langle \shift_Y f, g\rangle_{H_\Phi}
	&=
	\int_{\Bbb{C}^n} e(-\frac{1}{2}y\cdot\eta + \eta\cdot x)f(x-y)\overline{g(x)}e(2\ii \Psi(x, \ol{x}))\,(-2\ii)^{-n}\dd x \wedge \dd \overline{x}
	\\ &=
	\int_{\Bbb{C}^n} e(\frac{1}{2}y\cdot \eta + \eta\cdot x + 2\ii \Psi(x+y, \ol{x}))f(x)\ol{g(x)}(-2\ii)^{-n}\dd x \wedge \dd \overline{x}.
	\end{aligned}
\end{equation}
For $z \in \Bbb{C}^n$ fixed, as in the case of polynomials we can define $\ol{g}$ by $\ol{g(x-z)} = \ol{g}(\ol{x} - \ol{z})$. This is a function of the antiholomorphic variable $\ol{x}$, so a similar computation gives, for $Z = (z, \zeta) \in \Bbb{C}^{2n}$,
\begin{equation}
	\begin{aligned}
	\langle f, \shift_Z g\rangle_{H_\Phi}
	&=
	\int_{\Bbb{C}^n} f(x)\overline{e(-\frac{1}{2}z\cdot\zeta + \zeta \cdot x)g(x-z)}e(2\ii \Psi(x, \ol{x}))\,(-2\ii)^{-n}\dd x \wedge \dd \overline{x}
	\\ &=
	\int_{\Bbb{C}^n} e(-\frac{1}{2}\ol{z}\cdot \ol{\zeta} - \ol{\zeta}\cdot \ol{x} + 2\ii \Psi(x, \ol{x}+\ol{z}))f(x)\ol{g(x)}(-2\ii)^{-n}\dd x \wedge \dd \overline{x}.
	\end{aligned}
\end{equation}

One computes directly
\begin{equation}
	2\ii \Psi(x+y, \ol{x}) = 2\ii \Psi(x, \ol{x}) + \ii L y \cdot \ol{x} + \ii P y \cdot x + \frac{1}{2}\ii P y \cdot y
\end{equation}
and
\begin{equation}
	2\ii \Psi(x, \ol{x}+\ol{z}) = 2\ii\Psi(x, \ol{x}) + \ii \ol{L}\ol{z}\cdot x + \ii \ol{P}\ol{z}\cdot \ol{x} + \frac{1}{2}\ii \ol{P}\ol{z}\cdot \ol{z}.
\end{equation}

Therefore $\langle \shift_Y f, g\rangle_{H_\Phi} = \langle f, \shift_Z g\rangle_{H_\Phi}$ if and only if (equating coefficients of $(x, \ol{x})$)
\begin{equation}\label{eq:adj_deg1}
	\begin{pmatrix} \ii P & 1 \\ \ii L & 0\end{pmatrix} \begin{pmatrix} y \\ \eta \end{pmatrix}
	=
	\begin{pmatrix} \ii \ol{L} & 0 \\ \ii \ol{P} & -1 \end{pmatrix} \begin{pmatrix} \ol{z} \\ \ol{\zeta} \end{pmatrix}
\end{equation}
and (equating constant terms)
\begin{equation}\label{eq:adj_deg0}
	\frac{1}{2}y\cdot\eta + \frac{1}{2}\ii Py\cdot y = -\frac{1}{2}\ol{z}\cdot \ol{\zeta} + \frac{1}{2} \ii \ol{P}\ol{z}\cdot\ol{z}.
\end{equation}
Equation \eqref{eq:adj_deg1} is equivalent to $Z = -\ol{\Bff{A}_\Phi Y}$. Equation \eqref{eq:adj_deg0} is equivalent to
\begin{equation}
	y\cdot(\eta + \ii P y) = \ol{z}\cdot (\ii \ol{P}\ol{z} - \ol{\zeta}).
\end{equation}
When \eqref{eq:adj_deg1} holds, this is simply $y\cdot \ii \ol{L}\ol{z} = \ol{z} \cdot \ii L y$, which is automatic because $L$ is Hermitian. Therefore $Z = -\ol{\Bff{A}_\Phi Y}$ implies $\langle \shift_Y f, g\rangle = \langle f, \shift_Z g\rangle$ as claimed.
\end{proof}

Using Proposition \ref{prop:adjoint_shifts}, we recall some well-known facts about $\Bff{A}_\Phi$.

\begin{proposition}\label{prop:IPhi_algebra} 
Let $\Phi: \Bbb{C}^n \to \Bbb{R}$ be real-quadratic with $\Phi''_{\ol{x}x}$ positive definite, and recall $\Bff{A}_\Phi$ from \eqref{eq:def_IPhi}. Let
\begin{equation}\label{eq:def_xi_LambdaPhi}
	\Lambda_\Phi = \{(x, -2\ii\Phi'_x(x))\::\: x \in \Bbb{C}^n\}.
\end{equation}
Finally, let $\tilde{\barg}: L^2(\Bbb{R}^n) \to H_{\Phi}$ be any unitary FBI--Bargmann transform quantizing $\Bff{B}$ a complex linear canonical transformation. (Such a transformation exists by \eqref{eq:barg_unitary}; a metaplectic transformation like \eqref{eq:barg_meta} would work as well.)

Then $\Bff{B}(\Bbb{R}^{2n}) = \Lambda_\Phi$, $\Bff{A}_\Phi = \ol{\Bff{B}}\Bff{B}^{-1}$, and $X \mapsto \ol{\Bff{A}_\Phi X}$ is the unique antilinear involution preserving $\Lambda_\Phi$.
\end{proposition}

\begin{proof}
For $Y \in \Bbb{C}^{2n}$, the shift $\shift_Y$ is unitary on $H_\Phi$ if and only if $\tilde{\barg}^{-1}\shift_Y\tilde{\barg} = \shift_{\Bff{B}^{-1}Y}$ is unitary on $L^2(\Bbb{R}^n)$. The unitary shifts on $L^2(\Bbb{R}^n)$ correspond to real phase-space vectors, therefore $\shift_Y$ is unitary on $H_\Phi$ if and only if $\Bff{B}^{-1}Y \in \Bbb{R}^{2n}$. In other words, the set of unitary shifts on $H_\Phi$ is $\Bff{B}(\Bbb{R}^{2n})$. By Corollary \ref{cor:Lambda_Phi} and Proposition \ref{prop:adjoint_shifts}, $\Lambda_\Phi \subseteq \Bff{B}(\Bbb{R}^{2n})$. Since $-2\ii \Phi'_x(x)$ is a real-linear function of $x \in \Bbb{C}^n$, $\Lambda_\Phi$ as a vector space over $\Bbb{R}$ has real dimension $2n$, as does $\Bff{B}(\Bbb{R}^{2n})$. Therefore $\Lambda_\Phi = \Bff{B}(\Bbb{R}^{2n})$.

As for the relation between $\Bff{A}_\Phi$ and $\Bff{B}$, note that, as an operator on $H_\Phi$,
\begin{equation}
	\shift_X = \tilde{\barg}\tilde{\barg}^{-1} \shift_X \tilde{\barg}\tilde{\barg}^{-1} = \tilde{\barg} \shift_{\Bff{B}^{-1}X}\tilde{\barg}^{-1},
\end{equation}
where $\shift_{-\Bff{B}^{-1}X}$ acts on $L^2(\Bbb{R}^n)$. It is elementary that when $\shift_{Y}$ acts on $L^2(\Bbb{R}^n)$, $\shift_Y^* = \shift_{-\ol{Y}}$. Therefore the adjoint of $\shift_X$ on $H_\Phi$ can be computed as
\begin{equation}
	\tilde{\barg} \shift_{-\ol{\Bff{B}^{-1}X}}\tilde{\barg}^{-1} = \shift_{-\Bff{B}\ol{\Bff{B}^{-1}X}}.
\end{equation}
Comparing with Proposition \ref{prop:adjoint_shifts} gives $-\ol{\Bff{A}_\Phi X} = -\Bff{B}\ol{\Bff{B}^{-1}X}$ for all $X \in \Bbb{C}^{2n}$, so $\Bff{A}_\Phi = \ol{\Bff{B}}\Bff{B}^{-1}$.

The map $X \mapsto \ol{\Bff{A}_\Phi X}$ is transparently antilinear, and we see that it is an involution using Proposition \ref{prop:adjoint_shifts} and the fact that (as an operator on $H_\Phi$) $(\shift_X^*)^* = \shift_X$. Uniqueness of the antilinear involution comes from the fact that $\Lambda_\Phi$ has real dimension $2n$ and $\Lambda_\Phi \cap \ii\Lambda_{\Phi} = \{0\}$ (which in turn comes from $\Phi'_x(x) + \ii \Phi'_x(\ii x) = \ol{Lx}$ which vanishes only when $x = 0$). The fact that $\Lambda_\Phi$ is invariant under this involution follows from the fact that $X = \ol{\Bff{A}_\Phi X}$ if and only $\shift_X^* = \shift_{-X} = \shift_X^{-1}$ by Propostion \ref{prop:adjoint_shifts}. But we began by showing that the set of unitary shifts on $H_\Phi$ is precisely the set of shifts by elements of $\Lambda_\Phi$. Therefore $X \mapsto \ol{\Bff{A}_\Phi X}$ is the unique antilinear involution of $\Bbb{C}^n$ preserving $\Lambda_\Phi$, completing the proof of the proposition.
\end{proof}

Following \cite[Thm.~1.1]{Coburn_Hitrik_Sjostrand_2019}, we have that $\Phi_2 \geq \Phi_1$ if and only if we have the positivity condition
\begin{equation}\label{eq:positivity_IPhi}
	-\ii \left(\sigma(\Bff{A}_{\Phi_2}X, \ol{X}) - \sigma(\Bff{A}_{\Phi_1}X, \ol{X})\right) \geq 0, \quad \forall X \in \Bbb{C}^{2n}.
\end{equation}
For completeness, we adapt their proof in a specific case.

\begin{proposition}\label{prop:positivity}
For $j = 1, 2$, let $\Phi_j: \Bbb{C}^n \to \Bbb{R}$ be real-quadratic with $\Phi''_{\ol{x}x}$ positive definite, and let $\Bff{A}_{\Phi_j}$ be as in \eqref{eq:def_IPhi}. Then \eqref{eq:positivity_IPhi} holds if and only if $\Phi_2 \geq \Phi_1$. If $\Phi_2 > \Phi_1$ on $\Bbb{C}^n \backslash \{0\}$, then the inequality in \eqref{eq:positivity_IPhi} is strict when $X \neq 0$.
\end{proposition}

\begin{proof}
It suffices to prove the proposition in the case $\Phi_2 > \Phi_1$ on $\Bbb{C}^n \backslash \{0\}$, because having proved this case we can take the limit as $\eps\to 0^+$ of the proposition applied to $\Phi_2(x) + \frac{1}{2}\eps |x_0|^2$ in the place of $\Phi_2$.

To simplify notation, as we have seen in Example \ref{ex:Phi_reduction}, we can suppose that 
\begin{equation}
\Phi_2(x) = \Phi(x) = \frac{1}{2}(Lx\cdot \ol{x} + \jvRe(Px\cdot x))
\end{equation}
as in \eqref{eq:Phi_LP} and $\Phi_1(x) = \Phi_0(x) = \frac{1}{2}|x|^2$.

If $\Phi > \Phi_0$ on $\Bbb{C}^n \backslash\{0\}$, then $L > 1$ in the sense of positive definite matrices. This is because if there is some $x_0 \in \Bbb{C}^n \backslash \{0\}$ such that $Lx_0 \cdot \ol{x_0} \leq |x_0|^2$, then for an appropriate choice of $\theta \in \Bbb{R}$,
\begin{equation}
	\Phi(e(\theta)x_0) = \frac{1}{2}(Lx_0 \cdot \ol{x_0} - |Px_0\cdot x_0|) \leq \Phi_0(e(\theta)x_0),
\end{equation}
contradicting our assumption that $\Phi(x) > \Phi_0(x)$ for all $x \neq 0$.

Using \eqref{eq:IPhi_multiplied_out},
\begin{equation}
	-\ii\sigma(\Bff{A}_\Phi X, \ol{X}) = \begin{pmatrix} L - \ol{PL^{-1}}P & \ii \ol{PL^{-1}} \\ \ii \ol{L^{-1}}P & -\ol{L^{-1}}\end{pmatrix}X \cdot \ol{X}.
\end{equation}
Obtaining the same formula for $\Phi_0$ instead of $\Phi$ by setting $L = 1$ and $P = 0$ gives the following expression for the left-hand side of \eqref{eq:positivity_IPhi}:
\begin{equation}\label{eq:positivity_IPhi_IPhi0}
	-\ii(\sigma(\Bff{A}_\Phi - \Bff{A}_{\Phi_0})X, \ol{X}) = \begin{pmatrix} L - 1 - \ol{P L^{-1}} P & \ii\ol{P L^{-1}} \\ -\ii \ol{L^{-1}} P & 1-\ol{L^{-1}}\end{pmatrix} X \cdot \ol{X}.
\end{equation}
For this form to be positive definite (as claimed in the case of a strict inequality), it is necessary that $(1 - \ol{L^{-1}})\xi \cdot \ol{\xi} > 0$ for every $\xi \in \Bbb{C}^n \backslash\{0\}$. This is equivalent to $L > 1$ in the sense of positive definite matrices. We have therefore shown that if either $\Phi > \Phi_0$ on $\Bbb{C}^n \backslash \{0\}$ or if \eqref{eq:positivity_IPhi} holds strictly for $X \neq 0$, then $L > 1$. 

Invertibility of $1 - \ol{L^{-1}}$ allows us to complete the square in \eqref{eq:positivity_IPhi_IPhi0}: if $X = (x, \xi)$ and if
\begin{equation}
	\eta = -\ii (\ol{L}-1)^{-1} Px
\end{equation}
then
\begin{equation}\label{eq:positivity_complete_square}
	\begin{aligned}
	-\ii(\sigma((\Bff{A}_\Phi - \Bff{A}_{\Phi_0})X, \ol{X}) 
	&= (1-\ol{L}^{-1})(\xi + \eta)\cdot (\ol{\xi} + \ol{\eta})
	\\ &\qquad + (L-1-\ol{P L^{-1}}P - \ol{P}(\ol{L} - 1)^{-1}\ol{L^{-1}}P)x\cdot \ol{x}.
	\end{aligned}
\end{equation}
Positivity of \eqref{eq:positivity_IPhi_IPhi0} is therefore equivalent to positivity of
\begin{equation}
	\begin{aligned}
	(L-1 & - \ol{P}(1 + (\ol{L}-1)^{-1})\ol{L^{-1}}P)x\cdot \ol{x}
	\\ &=
	(L-1)x\cdot \ol{x} - (\ol{L}-1)^{-1}Px\cdot \ol{Px}.
	\end{aligned}	
\end{equation}
To simplify notation a little, we let 
\begin{equation}\label{eq:positivity_y_tildeP}
	y = (L-1)^{1/2}x, \quad \tilde{P} = (\ol{L}-1)^{-1/2}P(L-1)^{-1/2},
\end{equation}
so strict positivity for $X \neq 0$ of \eqref{eq:positivity_IPhi_IPhi0} is equivalent to strict positivity for $y \neq 0$ of
\begin{equation}\label{eq:positivity_normP}
	|y|^2 - |\tilde{P}y|^2.
\end{equation}

On the other hand,
\begin{equation}\label{eq:positivity_diffPhis}
	\begin{aligned}
	4(\Phi(x) - \Phi_0(x)) 
	&= 
	(L-1)x\cdot \ol{x} + \jvRe (Px\cdot x)
	\\ &=
	|y|^2 + \jvRe(\tilde{P}y\cdot y)
	\end{aligned}
\end{equation}
when $y$ and $\tilde{P}$ are as in \eqref{eq:positivity_y_tildeP}.

The reasoning of \cite[Eq.~(2.21)--(2.23)]{Coburn_Hitrik_Sjostrand_2019} shows that this is equivalent to positivity of \eqref{eq:positivity_normP} as follows. Positivity of \eqref{eq:positivity_normP} implies positivity of \eqref{eq:positivity_diffPhis} by the Cauchy-Schwarz inequality. Conversely, if $\Phi > \Phi_0$ on $\Bbb{C}^n \backslash \{0\}$, then multipliying $y$ by $e(\theta)$ in \eqref{eq:positivity_diffPhis} gives that there exists $c \in [0, 1)$ such that
\begin{equation}\label{eq:positivity_Phitilde_dot}
	|\tilde{P}y\cdot y| \leq c|y|^2, \quad \forall y \in \Bbb{C}^n.
\end{equation}
Therefore for all $z \in \Bbb{C}^n$,
\begin{equation}
	\begin{aligned}
	|\tilde{P}y \cdot z|
	&=
	\frac{1}{4}|\tilde{P}(y+z)\cdot(y+z) - \tilde{P}(y-z)\cdot(y-z)|
	\\ &\leq
	\frac{1}{4}c(|y+z|^2 + |y-z|^2) = \frac{1}{2}c(|y|^2 + |z|^2).
	\end{aligned}
\end{equation}
If neither $y$ nor $z$ is zero, we may replace $y$ by $y\sqrt{|z|/|y|}$ and $z$ by $z\sqrt{|y|/|z|}$ to obtain
\begin{equation}
	|\tilde{P}y\cdot z| \leq c|y|\,|z|
\end{equation}
for all $y, z \in \Bbb{C}^n$ (the case $y = 0$ or $z = 0$ being trivial). Therefore $\|P\|\leq c$ in operator norm which implies strict positivity of \eqref{eq:positivity_normP}. 

This completes the proof that strict positivity of \eqref{eq:positivity_diffPhis} and of \eqref{eq:positivity_normP} are equivalent. Up to changes of variables and completing the square, this is equivalent to the statement of the proposition in the case of a strict inequality. As mentioned at the beginning, the non-strict inequality can be obtained by taking the limit as $\eps \to 0^+$ of $\Phi_2(x) + \frac{1}{2}\eps|x|^2$, and the proof is therefore complete.
\end{proof}

\section{Proof of Theorem \ref{thm:main}}\label{s:proof_norm}

We will now prove Theorem \ref{thm:main}. Our main tool is \cite[Thm.~1.3]{Viola_2017}, which as mentioned therein is a straightforward consequence of \cite[Prop.~5.9,~Prop.~5.10]{Hormander_1995}. A version suitable for our purposes is the following.

\begin{theorem}\label{thm:metaplectic_sg_norm}
Let $\mathcal{K}$ be an element of the metaplectic semigroup (Definition \ref{def:metaplectic_semigroup}) quantizing a positive linear canonical transformation $\Bff{K}$. Then there there exist $\mu_1, \dots, \mu_n \in (0, 1]$ such that, counting for multiplicity, 
\begin{equation}\label{eq:embedding_norm_metasg}
	\operatorname{Spec} \ol{\Bff{K}^{-1}}\Bff{K} = \{\mu_j\}_{j=1}^n \cup \{\mu_j^{-1}\}_{j=1}^n
\end{equation}
and, as an operator on $L^2(\Bbb{R}^n)$,
\begin{equation}
	\|\mathcal{K}\| = \left(\prod_{j=1}^n \mu_j\right)^{1/4}.
\end{equation}
\end{theorem}

Suppose that $\Phi_2 \geq \Phi_1$. For $j = 1, 2$, let $\tilde{\barg}_j : L^2(\Bbb{R}^n) \to H_{\Phi_j}(\Bbb{C}^n)$ be unitary FBI--Bargmann transforms as in \eqref{eq:barg_unitary}. Then for any $f \in H_{\Phi_1} \backslash \{0\}$,
\begin{equation}\label{eq:proof_pullback_L2}
	\frac{\|f\|_{H_{\Phi_2}}}{\|f\|_{H_{\Phi_1}}} = \frac{\|\tilde{\barg}_2^{-1}f\|_{L^2(\Bbb{R}^n)}}{\|\tilde{\barg}_1^{-1}f\|_{L^2(\Bbb{R}^n)}} = \frac{\|\tilde{\barg}_2^{-1}\tilde{\barg}_1 g\|_{L^2(\Bbb{R}^n)}}{\|g\|_{L^2(\Bbb{R}^n)}},
\end{equation}
where $g = \tilde{\barg}_1^{-1} f$ can be, by varying $f$, any element of $L^2(\Bbb{R}^n)\backslash \{0\}$.

By \eqref{eq:barg_meta}, the metaplectic version
\begin{equation}\label{eq:proof_unitary_to_meta}
	\barg_2^{-1}\barg_1 = \left(\frac{\det L_1}{\det L_2}\right)^{-1/4}\tilde{\barg}_2^{-1}\tilde{\barg}_1
\end{equation}
satisfies \eqref{eq:meta_gaussian} and \eqref{eq:Egorov} for the complex linear canonical transformation $\Bff{B}_2^{-1}\Bff{B}_1$. Moreover, this canonical transformation is positive by Proposition \ref{prop:positivity} (which again is just a special case of \cite[Thm.~1]{Coburn_Hitrik_Sjostrand_2019}). Therefore $\barg_2^{-1}\barg_1$ belongs to the metaplectic semigroup (Definition \ref{def:metaplectic_semigroup}), and by Theorem \ref{thm:metaplectic_sg_norm}, there exist $\mu_1, \dots, \mu_n \in (0, 1]$ such that
\begin{equation}
	\operatorname{Spec}\ol{(\Bff{B}_2^{-1}\Bff{B}_1)^{-1}}\Bff{B}_2^{-1}\Bff{B}_1 = \{\mu_j\}_{j=1}^n \cup \{\mu_j^{-1}\}_{j=1}^n
\end{equation}
and
\begin{equation}\label{eq:proof_norm_meta_sg}
	\|\barg_2^{-1}\barg_1\| = \left(\prod_{j=1}^n \mu_j\right)^{1/4}.
\end{equation}

The spectrum of a matrix is unchanged by a similarity transform, so using Proposition \ref{prop:IPhi_algebra},
\begin{equation}
	\operatorname{Spec} \ol{(\Bff{B}_2^{-1}\Bff{B}_1)^{-1}}\Bff{B}_2^{-1}\Bff{B}_1 = \operatorname{Spec} \Bff{B}_1\ol{\Bff{B}_1^{-1}\Bff{B}_2^{-1}}\Bff{B}_2^{-1} = \operatorname{Spec} \Bff{A}_{\Phi_1}^{-1} \Bff{A}_{\Phi_2}.
\end{equation}
Inserting into \eqref{eq:proof_norm_meta_sg} and using \eqref{eq:proof_pullback_L2} and \eqref{eq:proof_unitary_to_meta} gives the result of Theorem \ref{thm:main} when $\Phi_2 \geq \Phi_1$.

Suppose now that there exists some $x_0 \in \Bbb{C}^n \backslash \{0\}$ such that $\Phi_1(x_0) > \Phi_2(x_0)$. By Example \ref{ex:Phi_reduction} we may assume without loss of generality that $\Phi_1(x) = \Phi_0(x) = \frac{1}{2}|x|^2$, and furthermore by scaling we may assume that $|x_0| = 1$. Consider
\begin{equation}
	h_\delta(x) = e(-\frac{1}{2}\ii \delta (\ol{x_0}\cdot x)^2), \quad \delta \in [0, 1).
\end{equation}
Notice that
\begin{equation}
	|h_\delta(x)|^2 e(2\ii\Phi_0(x)) = e(2\ii(\Phi_0(x) - \frac{1}{2}\delta \jvRe((x_0 \cdot x)^2)) = \ee^{-2\pi(|x|^2 - \delta \jvRe((x_0 \cdot x)^2))}.
\end{equation}
Since 
\begin{equation}
	\delta \jvRe((x_0 \cdot x)^2) \leq \delta |x_0|^2|x|^2 = \delta |x|^2, 
\end{equation}
when $\delta \in [0, 1)$ we have that $|h_\delta(x)|^2 e(2\ii\Phi_0(x))$ is the exponential of a negative definite quadratic form which is therefore integrable. This shows that $h_\delta \in H_{\Phi_0}$.

On the other hand,
\begin{equation}
	|h_\delta(x_0)|^2 e(2\ii\Phi_2(x_0)) = e(2\ii(\Phi_2(x_0) - \frac{1}{2}\delta |x_0|^2)) = \ee^{4\pi (\delta \Phi_0(x_0) - \Phi_2(x_0))}.
\end{equation}
Therefore $|h_\delta(x)|^2 e(2\ii\Phi_2(x))$ is again the exponential of a quadratic form, but for $\delta$ sufficiently close to $1$ this quadratic form is not even negative semidefinite. Therefore, for $\delta$ near $1$, $h_\delta \notin H_{\Phi_2}$. We have exhibited (up to changes of variables in Example \ref{ex:Phi_reduction} taking $\Phi_1$ to $\Phi_0$) an explicit element of $H_{\Phi_1}$ for which $\|h_\delta\|_{H_{\Phi_2}} = \infty$, meaning that the embedding between these two spaces cannot be bounded.

This shows that $\Phi_2 \geq \Phi_1$ is a necessary condition for $\iota : H_{\Phi_1} \to H_{\Phi_2}$ to be bounded. Above, we proved the formula for the norm in the case $\Phi_2 \geq \Phi_1$, and the proof of Theorem \ref{thm:main} is therefore complete.

\section{The Gaussian witnessing the maximum norm ratio}\label{s:proof_where}

We now prove Theorem \ref{thm:norm_where}.

A straightforward consequence of \cite[Prop.~5.9,~Prop.~5.10]{Hormander_1995} is that if $\mathcal{K}:L^2(\Bbb{R}^n) \to L^2(\Bbb{R}^n)$ is an element of the metaplectic semigroup quantizing $\Bff{K}$ a \emph{strictly} positive complex linear canonical transformation, then $\mathcal{K}^*\mathcal{K}$ quantizes $\ol{\Bff{K}^{-1}}\Bff{K}$ and, for some $q(x, \xi)$ a real-valued positive definite quadratic form on $\Bbb{R}^{2n}$,
\begin{equation}
	\mathcal{K}^*\mathcal{K} = \exp(-2\pi q^w(x, D_x)).
\end{equation}
(Here $q^w(x, D_x)$ is the Weyl quantization of $q$.)

It is well-known that the ground state of $q^w(x, D_x)$ is $g_{T_0}(x) = e(\frac{1}{2}T_0x \cdot x)$ where the symmetric matrix (with positive definite imaginary part) $T_0$ is defined by
\begin{equation}
	\{(x, T_0x)\}_{x \in \Bbb{C}^n} = \bigoplus_{-\ii \lambda > 0} \ker(H_q - \lambda).
\end{equation}
Here, $H_q$ is the Hamilton map of $q$, the unique matrix antisymmetric with respect to the symplectic form $\sigma$ such that $q(X) = \sigma(X, H_q X)$. (See for instance \cite[Thm.~3.5]{Sjostrand_1974}.)

By the exact classical-quantum correspondence \cite[Thm.~5.12]{Hormander_1995}, $\exp(-2\pi q^w) = e(\ii q^w)$ is an element of the metaplectic semigroup quantizing $\exp(H_{-\ii q}) = \exp(-\ii H_q)$. Having chosen $q$ such that $\mathcal{K}^* \mathcal{K} = \exp(-q^w)$ and checking that $\mathcal{K}^*$ quantizes $\ol{\Bff{K}^{-1}}$, we conclude that
\begin{equation}
	\exp(H_{-\ii q}) = \ol{\Bff{K}^{-1}}\Bff{K}.
\end{equation}
Equivalently, $\exp(\ii H_q) = \Bff{K}^{-1}\ol{\Bff{K}}$. We insert this fact into the definition of $T_0$, using the notation $\mu = \exp(\ii \lambda)$:
\begin{equation}
\begin{aligned}
	\{(x, T_0 x)\}_{x \in \Bbb{C}^n} 
	&=
	\bigoplus_{-\ii \lambda > 0} \ker(\ii H_q - (\ii \lambda))
	\\ &=
	\bigoplus_{-\ii \lambda > 0} \ker(\exp(\ii H_q) - \ee^{\ii\lambda})
	=
	\bigoplus_{\mu < 1} \ker(\Bff{K}^{-1}\ol{\Bff{K}} - \mu).
\end{aligned}
\end{equation}

We obtained the norm of the embedding from $H_{\Phi_1}$ to $H_{\Phi_2}$ via the observation that, when $\tilde{\barg}_j : L^2(\Bbb{R}^n) \to H_{\Phi_j}$ are unitary FBI--Bargmann transforms and when $g = \tilde{\barg}_1^{-1}f$,
\begin{equation}\label{eq:optimizing_chgofvars}
	\frac{\|f\|_{H_{\Phi_2}}}{\|f\|_{H_{\Phi_1}}} = \frac{\|\tilde{\barg}_2^{-1}\tilde{\barg}_1 g\|_{L^2}}{\|g\|_{L^2}},
\end{equation}
when $g = \barg_1^{-1}f$. When $\Phi_2(x) > \Phi_1(x)$ for all $x \neq 0$, $\tilde{\barg}_2^{-1}\tilde{\barg}_1$ quantizes a strictly positive transformation following Proposition \ref{prop:positivity}. Then we may find $g_{T_0}$ optimizing the right-hand side of \eqref{eq:optimizing_chgofvars} by applying the preceding discussion to $\Bff{K} = \Bff{B}_2^{-1}\Bff{B}_1$.

The link between $\{(x, T_0x)\}$ and $g_{T_0}(x) = e(\frac{1}{2}T_0 x\cdot x)$ comes from
\begin{equation}
	g \in \operatorname{Span}\{g_{T_0}\} \iff \forall x \in \Bbb{C}^n, (\shift_{(x, T_0 x)} - 1)g = 0.
\end{equation}
We can identify $g_T = \tilde{\barg}_1 g_{T_0}$ up to a constant by applying $\tilde{\barg}_1$ to both sides and using the Egorov relation \eqref{eq:Egorov}:
\begin{equation}
	(\shift_{(x, T_0 x)} - 1)g = 0 \iff (\shift_{\Bff{B}_1(x, T_0 x)} - 1)\tilde{\barg}_1 g = 0.
\end{equation}
Therefore we are looking for $T$ such that
\begin{equation}
	\{(x, Tx)\}_{x \in \Bbb{C}^n} = \{\Bff{B}_1(x, T_0 x)\}_{x \in \Bbb{C}^n}.
\end{equation}
It is not automatic that $\{\Bff{B}_1(x, T_0 x)\}$ will take the form of a graph, but examining \eqref{eq:barg_gT} we see that the standard FBI--Bargmann transform takes Gaussians $g_{T_0}(x) = e(\frac{1}{2}T_0 x \cdot x)$ with $\jvIm T_0 > 0$ to Gaussians $g_T(x)$ where $\jvIm T$ is not necessarily positive definite. Examining \eqref{eq:barg_unitary}, this continues to be true for any FBI--Bargmann transform we are considering.

We observe that 
\begin{equation}
	(\ol{\Bff{K}^{-1}}\Bff{K} - \mu) X = 0 \iff (\Bff{B}_1\ol{\Bff{K}^{-1}}\Bff{K}\Bff{B}_1^{-1} - \mu) \Bff{B}_1 X = 0.
\end{equation}
Explicitly, since $\Bff{K} = \Bff{B}_2^{-1}\Bff{B}_1$,
\begin{equation}
	\Bff{B}_1\Bff{K}^{-1}\ol{\Bff{K}}\Bff{B}_1^{-1} = \Bff{B}_1 \Bff{B}_1^{-1} \Bff{B}_2 \ol{\Bff{B}_2^{-1}\Bff{B}_1}\Bff{B}_1^{-1} = \Bff{A}_{\Phi_2}^{-1}\Bff{A}_{\Phi_1}
\end{equation}
by Proposition \ref{prop:IPhi_algebra}.

Therefore 
\begin{equation}
	\{(x, Tx)\}_{x \in \Bbb{C}^n} = \{\Bff{B}_1(x, T_0 x)\}_{x \in \Bbb{C}^n} = \bigoplus_{\mu < 1} \ker(\Bff{A}_{\Phi_2}^{-1} \Bff{A}_{\Phi_1} - \mu).
\end{equation}
Since $\tilde{\barg}_1 g_{T_0} = c g_{T}$ for some $c \in \Bbb{C}$, from
\begin{equation}
	\|\tilde{\barg}_2^{-1}\tilde{\barg}_1\| = \frac{\|\tilde{\barg}_2^{-1} \tilde{\barg}_1 g_{T_0}\|_{L^2}}{\|g_{T_0}\|_{L^2}}
\end{equation}
we get
\begin{equation}
	\|\tilde{\barg}_2^{-1}\tilde{\barg}_1\| = \frac{\|g_T\|_{H_{\Phi_2}}}{\|g_T\|_{H_{\Phi_1}}}.
\end{equation}
By \eqref{eq:proof_pullback_L2}, $\|\tilde{\barg}_2^{-1}\tilde{\barg}_1\|$ (as an operator on $L^2(\Bbb{R}^n)$) is equal to the norm of the embedding from $H_{\Phi_1}$ to $H_{\Phi_2}$. The proof of Theorem \ref{thm:norm_where} is therefore complete.

\bibliographystyle{plain}
\bibliography{norm_embeddings}

\end{document}